\documentclass[letterpaper,12pt]{article}
\usepackage{mathptmx}
\usepackage[authoryear,comma,longnamesfirst,sectionbib]{natbib} 
\usepackage[normalem]{ulem}
\usepackage{multirow}

\usepackage{amsmath,amsthm,amssymb,graphicx,longtable}

\usepackage{color}

\theoremstyle{plain}
\newtheorem{thm}{Theorem}
\newtheorem{lem}[thm]{Lemma}
\newtheorem{prop}[thm]{Proposition}

\theoremstyle{definition}
\newtheorem*{defn}{Definition}

\newtheorem*{rem}{Remark}

\newcommand{\RR}{\mathbb R}

\newcommand{\diag}{\operatorname{diag}}
\newcommand{\pa}{\operatorname{pa}}

\newcommand{\dsp}{\displaystyle}

\title{Parameter identifiability of discrete Bayesian networks with hidden variables}

\author{
{Elizabeth S.~Allman}\\
Department of Mathematics and Statistics \\
University of Alaska Fairbanks\\
\and
{John A.~Rhodes}\\
Department of Mathematics and Statistics \\
University of Alaska Fairbanks\\
\and
{Elena Stanghellini}\\
Dipartimento di Economia Finanza e Statistica\\
Universit\`a di Perugia\\
\and
{Marco Valtorta}\\
Deptartment of Computer Science and Engineering \\
University of South Carolina\\
}

\date{June 2, 2014}

\newcommand{\Addresses}{
\bigskip
\footnotesize

E.S.~Allman,\textsc{
Department of Mathematics and Statistics,
University of Alaska Fairbanks, Fairbanks, AK 99775, USA}, \texttt{esallman@alaska.edu}

\medskip 

J.A.~Rhodes,\textsc{
Department of Mathematics and Statistics,
University of Alaska Fairbanks, Fairbanks, AK 99775, USA}, \texttt{j.rhodes@alaska.edu}

\medskip

E.~Stanghellini,\textsc{
Dipartimento di Economia Finanza e Statistica,
Universit\`a di Perugia, 06100 Perugia, Italy}, \texttt{elena.stanghellini@stat.unipg.it}

\medskip

M.~Valtorta,\textsc{
Department of Computer Science and Engineering,
University of South Carolina, Columbia, SC 29208, USA}, \texttt{MGV@cse.sc.edu}

}

\begin{document}

\maketitle

\begin{abstract}
Identifiability of parameters is an essential property for a statistical model to be useful in most settings. However,
establishing parameter identifiability for Bayesian networks with hidden variables remains challenging.  In the context of finite state spaces, we give algebraic
arguments establishing identifiability of some special models on small DAGs. We also establish that, for fixed state spaces, generic identifiability of parameters depends only on the Markov equivalence class of the DAG.
To illustrate the use of these results, we investigate identifiability for all binary Bayesian networks with up to five variables, one of which is hidden and parental to all observable ones.
Surprisingly, some of these models have parameterizations that are generically 4-to-one, and not 2-to-one as label swapping of the hidden states would suggest. This leads to interesting difficulties in interpreting causal effects.
\end{abstract}

\section{Introduction}

A Directed Acyclic Graph (DAG) can represent the factorization of a joint distribution of a set of random variables. To be more precise, a Bayesian network is a pair (G,P), where G is a DAG and P is a joint probability distribution of variables in one-to-one correspondence with the nodes of G, with the property that each variable is conditionally independent of its non-descendants given its parents.  It follows from this definition that the joint probability P factors according to G, as the product of the conditional probabilities of each node given its parents. Thus a discrete Bayesian network is fully specified by a DAG and a set of conditional probability tables, one for each node given its parents \citep{Neapolitan90, Neapolitan04}. 

A causal Bayesian network is a Bayesian network enhanced with a causal interpretation. Work initiated by \citet{Pearl1995,Pearl2nd} investigated the
identification of causal effects in causal Bayesian networks when 
some variables are assumed
observable and others are hidden. In a non-parametric setting,
with no assumptions about the state space of variables, 
there is a
complete algorithm for determining which causal effects between
variables are identifiable \citep{Huang2006, Shpitser2008, TianP02a,
Pearl2012}.

\begin{figure}[h]\label {fg:Kuroki}
\begin{center}
\includegraphics[height=.07\textheight]{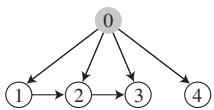}
\end{center}
\caption{The DAG of a Bayesian network studied by \citet{Kuroki2014}, denoted 4-2b in the Appendix. }
\end{figure}

As powerful as this theory is, however, it does not address
identifiability when assumptions are made on the nature of
the variables. Indeed, by specializing to finite state
spaces, causal effects that were non-identifiable according to the
theory above may become identifiable. 
One particular example, with DAG shown in Figure \ref{fg:Kuroki},
has been studied by \citet{Kuroki2014}.
If the state space of hidden variable 0 is finite, and observable variables 1 and 4 have state spaces of larger sizes,
then the causal effect of variable 2 on variable 3 can be determined, for generic parameter choices. 

\medskip

In this paper we study in detail identification properties of certain small Bayesian networks, as a first step toward developing a systematic understanding of identification in the presence of finite hidden variables. While this includes an analysis of the model with the DAG above, our motivation is different from that of \citet{Kuroki2014}, and results were obtained independently.
We make a thorough study of networks with up to five binary variables, one of which is unobservable and parental to all observable
ones, as shown in Table \ref{tb:results} of the Appendix. This leads us to develop some basic tools and arguments that can be applied more generally to questions of parameter identifiability.
Then, for each such binary model, we determine a value $k\in \mathbb N\cup\{\infty\}$ such that the marginalization from the full joint distribution to that over the observable variables is generically $k$-to-one. 
Although we restrict this exhaustive study to binary models for simplicity, straightforward modifications to our arguments would extend them to larger state spaces.
A typical requirement for such an extended identifiability result is that the state spaces of observable variables be sufficient large, relative to that of the hidden variable, as in the result of  \citet{Kuroki2014} described above. (In particular, that result restricted to finite state spaces follows easily from our framework, and can be obtained for continuous state spaces of observable variables using arguments of \cite{AMR09}.)

\medskip

We use the term ``DAG model"  for the collection of all Bayesian networks with the same DAG and specification of state spaces for the variables. With the conditional probability tables of nodes given their parents forming the parameters of the model,
we thus allow these tables to range over all valid tables of a fixed size to give the parameter space of such a model.

In dealing with discrete unobserved variables, one well-understood identifiability issue is sometimes called \emph{label swapping}. If the latent variable has $r$ states, there are $r!$ parameter choices, obtained by permuting the state labels of the latent variable, that generate the same observable distribution. Thus the parametrization map is generically at least $r!$-to-one. For models with a single binary latent variable, it is thus commonly expected that parameterizations are either infinite-to-one due to a parameter space of too high a dimension, or 2-to-one due to label swapping. Our work, however, finds surprisingly simple examples such that the mapping is 4-to-one, so that more subtle non-identifiability issues arise. The implications of this for determining causal effects are also explored. 

Our analysis arises from an algebraic viewpoint of the identifiability problem.
With finite state spaces the parameterization maps for DAG models with hidden variables are polynomial. Given a distribution arising from
the model, the parameters are identifiable precisely when a certain system of multivariate
polynomial equations has exactly one solution  (up to label-swapping of states for hidden variables).
Though in principle computational algebra software can be used to
investigate parameter identifiability, the necessary calculations are usually intractable for even moderate size DAGs and/or state spaces.
In addition, one runs into issues of complex versus real roots, and the difficulty of
determining when real roots lie within stochastic bounds. While our arguments
are fundamentally algebraic, they do not depend on any machine
computations.

If a single polynomial $p(x)$ in one variable is given, of degree $n$, then it is well
known that the map from $\mathbb C$ to $\mathbb C$ that it defines will be
generically $n$-to-one. Indeed the equation $p(x)=a$ will be of degree $n$
for each choice of $a$, and generically will have $n$ distinct roots. This fact
generalizes to polynomial maps from $\mathbb C^n$ to $\mathbb C^m$;
there always exists a $k\in \mathbb N\cup\{\infty\}$ such that the map is generically $k$-to-one.

However if $p(x)$ has real coefficients, and is instead viewed as a map from (a subset of)
$\mathbb R$ to $\mathbb R$, it may not have a generic $k$-to-one behavior. For instance,
from a typical graph of a cubic one sees there can be a sets of positive
measure on which it is 3-to-one, and others on which it is one-to-one, as well as an exceptional
set of measure zero on which the cubic is 2-to-one.
While this exceptional set arises since a
polynomial may have repeated roots, the lack of a generic $k$-to-one behavior is due to
passing from considering a complex domain for the function, to a real one.

The fact that the polynomial parameterizations for the models investigated here have a
generic $k$-to-one behavior on their parameter space thus depends on the particular form of the parameterizations.
For those binary models in Table \ref{tb:results}, we prove this essentially one model at a time, while obtaining the value
for $k$. In the case of finite $k$, our arguments actually go further and characterize the $k$
elements of $\phi^{-1}(\phi(\theta))$ in terms of a generic $\theta$. Of course
when $k=2$ this is nothing more than label swapping, but for the cases of $k=4$
more is required. Precise statements appear in later sections. In some cases,
we also give descriptions of an exceptional subset of $\Theta$
where the generic behavior may not hold.
In all cases, the reader can deduce such a set from our arguments.

\medskip

After setting terminology in Section \ref{sec:defs}, in Section \ref{sec:markovequiv} we establish that, when all variables have fixed finite state spaces, Markov equivalent DAGs specify parameter equivalent models. More specifically, there is a invertible rational map between generic parameters on the DAGs which lead to the same distributions. Thus in answering generic identifiability questions one need only consider Markov equivalence classes of DAGs.
In Section \ref{sec:special} we revisit the fundamental result due to \citet{Kruskal77}, as developed in \citet{AMR09} for identifiability questions. We give an explicit identifiability procedure for the DAG it most directly applies to. We also use our proof  technique for this explicit Kruskal result to obtain an identification procedure for a different specific DAG. These two DAGs are basic cases whose known identifibility can be leveraged to study other models.

In Section \ref{sec:binary} these general theorems, combined with auxiliary arguments, are enough to determine generic identifiability of all the binary DAG models we catalog. Although we do not push these arguments toward exhaustive consideration of non-binary models here, in many cases it would be straightforward to do so. For instance if all variables associated to a DAG have the same size state space, little in our arguments needs to be modified. For models in which different variables have different size finite state spaces, one must be more careful, but many generalizations are fairly direct. Finally in Section \ref{sec:4to1} we investigate the implications of the generically 4-to-one parameterization uncovered for one of these models.

\medskip

We view the main contribution of this paper not as the determination of parameter
identifiability for the specific binary models we consider, but rather as the development of the techniques
by which we establish our results. We believe these examples will lead to a more general
understanding of identifiability for finite state DAG models.
Ultimately, one
would like fairly simple graphical rules to determine which
parameters are identifiable, and perhaps even to yield formulas for
them in terms of the joint distribution. While it is unclear to what
extent this is possible, even partial results covering only certain
classes of DAGs, or some state spaces, would be useful.

Ultimately, establishing similar results for more general graphical models, not specified by a DAG, would be desirable. Some work in this context already exists; see, for example, \citet{Stanghellini2013}. However, in both the DAG and more general setting, investigations are still at a rudimentary level.

\section{Discrete DAG models and parameter identifiability}\label{sec:defs}
The models we consider are specified in part by DAGs  $\mathcal G=(V,E)$ in which nodes $v\in V$
represent random variables $X_v$, and directed edges in $E$ imply
certain independence statements for the joint distribution of all
variables \citep{LauritzenBook}. A bipartition of $V=O\sqcup H$ is
given, in which variables associated to nodes in $O$ or $H$ are
observable or hidden, respectively. Finally, we fix finite state
spaces, of size $n_v$ for each variable $X_v$.

A DAG $\mathcal G$ entails a collection of conditional independence
statements on the variables associated to its nodes, via d-separation,
or an equivalent separation criterion in terms of the moral graph
on ancestral sets.
A joint distribution of variables satisfies these statements precisely when it has a factorization according
to $\mathcal G$ as  $$P=\prod_{v\in V} P(X_v|X_{\pa(v)}),$$ with
$\pa(v)$ denoting the set of parents of $v$ in $\mathcal G$. We refer to
the conditional probabilities $\theta=(P(X_v|X_{\pa(v)}))_{v\in V}$
as the \emph{parameters} of the DAG model, and denote the space of all
possible choices of parameters by $\Theta=\Theta_{\mathcal G,\{n_v\}}$. The
parameterization map for the joint distribution of all variables,
both observable and hidden, is denoted
$$\phi:\Theta\to\Delta^{(\prod_{v\in V} n_v)-1},$$ where $\Delta^{k}$ is the $k$-dimensional
probability simplex
of stochastic vectors in $\mathbb R^{k+1}$. Thus $\phi(\Theta)$ is
precisely the collection of all probability distributions satisfying
the conditional independence statements associated to $\mathcal G$ (and possibly additional ones).

Since the probability distribution for the model with hidden
variables is obtained from that of the fully observable model, its
parameterization map is
$$\phi^+=\sigma\circ\phi:\Theta\to\Delta^{(\prod_{v\in O} n_v)-1},$$
where $\sigma$ denotes the appropriate map marginalizing over hidden
variables. The set $\phi^+(\Theta)$ is thus the collection of all
observable distributions that arise from the hidden variable model. This
collection depends not only on the DAG and designated state spaces
of observable variables, but also on the state spaces of hidden
variables, even though the sizes of hidden state spaces are not
readily apparent from an observable joint distribution.

With all variables having finite state spaces, the parameter space
$\Theta$ can be identified with the closure of an open subset of
$[0,1]^L$, for some $L$. We refer to $L$ as the dimension of
the parameter space. The dimension of $\Theta$ is easily seen to be
\begin{equation}
\dim(\Theta)=\sum_{ v\in V}\left (( n_v-1)\prod_{w\in \pa(v)} n_w\right ).
\end{equation}
In the case of all binary variables, this simplifies to
\begin{equation}
\dim(\Theta)=\sum_{ v\in V}2^{|\pa(v)|}=\sum_{k=0}^\infty m_k2^k, \label{eq:dimtheta}
\end{equation}
where $m_k$ is the number of nodes in $\mathcal G$ with in-degree $k$.
\medskip

If a statement is said to hold for \emph{generic parameters}  or
\emph{generically} then we mean it holds for all parameters in a set
of the form $\Theta\smallsetminus E$, where the exceptional set $E$
is a proper algebraic subset of $\Theta$. (Recall an \emph{algebraic
subset}  is the zero set of a finite collection of multivariate polynomials.) As
proper algebraic subsets of $\RR^n$ are always of Lebesgue measure
zero, a statement that holds generically can fail only on a set of
measure zero.

As an example of this language,  for any DAG model with all variables
finite and observable, generic parameters lead to a distribution faithful
to the DAG, in the sense that those conditional independence statements
implied by d-separation rules will hold, and no others \citep{Meek95}.
Equivalently, a generic distribution from such a model is faithful to the DAG.

\medskip

There are several notions of identifiability of parameters of a
model; we refer the reader to \citet{AMR09}.  The strictest notion, that the parameterization map is one-to-one, is easily seen to hold when all DAG variables are observable with mild additional assumptions (\emph{e.g.}, positivity of all parameters).  
If a model has hidden variables, then this is too strict a notion of identifiability, as the well-known issue of label swapping arises: One can permute the names of the states of hidden
variables, making appropriate changes to associated parameters, without
changing the joint distribution of the
observable variables. 
For a model with one $r$-state  hidden
variable, label swapping implies that for any generic $\theta_1\in \Theta$ there are at least $r!-1$ other points $\theta_j\in \Theta$ with
$\phi^+(\theta_1)=\phi^+(\theta_j)$. But since these are
isolated parameter points that differ only by state labeling, this issue does not generally limit the usefulness of a model, provided that we remain aware of it when interpreting parameters.

The strongest useful notion of identifiability for models with
hidden variables is that for generic $\theta_1\in \Theta$, if
$\phi^+(\theta_1)=\phi(\theta_2)^+$, then $\theta_1$ and $\theta_2$
differ only up to label swapping for hidden variables. This notion
is our primary focus in this paper, which we refer to it as
\emph{generic identifiability up to label swapping}. In particular,
for models with a single binary hidden variable it is equivalent to
the parameterization map being generically 2-to-one.

\section{Markov equivalence and parameter identifiability}\label{sec:markovequiv}

Two DAGs on the same sets of observable and hidden nodes are said to
be \emph{Markov equivalent} if they entail the same conditional
independence statements
through d-separation. (Note this notion does not distinguish between
observable and hidden variables; all are treated as
observable.) Thus for fixed choices of state spaces of the variables,
two different
but Markov equivalent DAGs, $\mathcal G_1\cong \mathcal G_2$, have
different parameter spaces $\Theta_1, \Theta_2,$ and different
parameterization maps, yet $\phi_1(\Theta_1)=\phi_2(\Theta_2)$.

\medskip

For studying identifiability questions, it is helpful to first explore the relationship between parameterizations for Markov equivalent graphs.
A simple example, with no hidden variables, is instructive. Consider the DAGs on two observable nodes 
$$1\to 2, \ \ \ \ 1\leftarrow 2,$$  which are equivalent, since neither entails any independence statements. Now the particular probability distribution $P(X_1=i,X_2=j)=P_{ij}$ with 
$$P=\begin{pmatrix} 1/2&0\\1/2&0\end{pmatrix}$$ requires parameters on the first DAG to be $$P(X_1)=(1/2,1/2),\ P(X_2|X_1)=\begin{pmatrix}1&0\\1&0\end{pmatrix},$$ while parameters on the second DAG can be $$P(X_2)=(1,0),\ P(X_1|X_2)=\begin{pmatrix}1/2&1/2\\t&1-t\end{pmatrix}$$ for any $t\in[0,1]$. Thus this particular distribution has identifiable parameters for only one of these  DAGs. (Here and in the rest of the paper conditional probability tables specifying parameters have rows corresponding to states of conditioning, i.e. parent, variables.) 

Of course, this probability distribution was a special one, and is atypical for these models, which are easily seen to have  generically identifiable parameters (as do all  DAG models without hidden variables). Nonetheless, it illustrates the need for `generic' language and careful arguments for results such as the following.

\begin{thm}\label{thm:equivDAGs}
With all variables having fixed finite state spaces, consider two Markov equivalent
DAGs, $\mathcal G_1$ and $\mathcal G_2$, possibly with hidden nodes. If the
parameterization map $\phi_1^+$ is generically $k$-to-one for some
$k\in\mathbb N$, then $\phi_2^+$ is also generically $k$-to-one.

In particular if such a model has parameters that are generically identifiable up to
label swapping, so does every Markov equivalent model.
\end{thm}

This theorem is a consequence of the following:

\begin{lem}\label{lem:homeo}
With all variables having finite state spaces, consider two Markov equivalent
DAGs, $\mathcal G_1$ and $\mathcal G_2$, with
parameter spaces $\Theta_i$
and parameterization maps $\phi_i$, $i \in \{1,2\}$, for the joint distribution of all variables. Then there
are generic subsets $S_i\subseteq\Theta_i$ and a
rational homeomorphism  $\psi:S_1\to S_2$, with rational inverse,
such that for all $\theta\in S_1$
$$\phi_1(\theta)=\phi_2(\psi(\theta)).$$
\end{lem}

\begin{proof}
Recall that an edge $i\to j$ of a DAG is said to be covered if $\pa(j)=\pa(i)\cup\{i\}$.
By \citet{Chickering95}, Markov equivalent DAGs differ by
applying a sequence of reversals of covered edges.

We thus first assume the $\mathcal G_i$ differ by the reversal of a
single covered edge $i\to j$ of $\mathcal G_1$.
Let $W=\pa_{\mathcal G_1}(i)=\pa_{\mathcal G_2}(j)$, so
$\pa_{\mathcal G_1}(j)=W\cup\{i\}$, $\pa_{\mathcal G_2}(i)=W\cup\{j\}$.
Now any $\theta\in \Theta_1$ is a collection of conditional probabilities
$P(X_v|X_{\pa(v)})$, including $P(X_i|W), P(X_j|X_i,W)$.
From these, successively define
\begin{align*}P(X_i,X_j|W)&=P(X_j|X_i,W)P(X_i|W),\\
P(X_j|W)&=\sum_{k}P(X_i=k,X_j|W),\\
P(X_i | X_j, W)&=P(X_i,X_j | W)/ P(X_j|W).
\end{align*}
Using these last two conditional probabilities, along with those
specified by $\theta$ for all $v\ne i,j$, define
parameters $\psi(\theta)\in \Theta_2$.
Now $\psi$ is defined and continuous on  the set
$S_1$
where
$P(X_i|W)$ and $P(X_j|X_i,W)$ are strictly positive.

One easily checks that the same construction applied
to the edge $j\to i$ in $\mathcal G_2$ gives the inverse map.

If $\mathcal G_1,\mathcal G_2$ differ by a sequence of
edge reversals, one defines the $S_i$ as subsets where
all parameters related to the reversed edges are strictly positive, and
let $\psi$ be the composition of the maps for the individual reversals.
\end{proof}

\begin{proof}[Proof of Theorem \ref{thm:equivDAGs}]  
Suppose $\Theta_1$ has a generic subset $S$ on which both $\phi_1^+$ is $k$-to-one and the map $\psi$ of Lemma \ref{lem:homeo} is invertible.
Then  $\psi(S)$ will be a generic subset of  $\Theta_2$, and 
the identity
$$\phi^+_2(\theta)=\phi^+_1(\psi^{-1}(\theta))$$
from Lemma \ref{lem:homeo} shows that $\phi^+_2$ is $k$-to-one on $\phi(S)$.
Thus we need only establish the existence of such an $S$.

Let $S_1=\Theta_1\smallsetminus E_1$,
$S_2=\Theta_2\smallsetminus E_2$ be the generic sets of Lemma  \ref{lem:homeo}. Let $S_1'=\Theta_1\smallsetminus E_1'$
be a generic set on which $\phi_1^+$ is $k$-to-one. We may thus assume $E_1,E_1', E_2$ are all proper algebraic subsets.
Since $\phi^+_1$ is generically $k$-to-one with finite $k$, the set $(\phi_1^+)^{-1}(\phi_1^+(E_1))$ must be contained in a proper algebraic subset of $\Theta_1$, say $E_1''$.
We may therefore take $S=\Theta_1\smallsetminus (E_1'\cup E_1'')$.
\end{proof}

\section{Two special models}\label{sec:special}

In this section, we explain how one may explicitly solve for parameter
values from a joint distribution of the observable variables for models specified by two specific DAGs with hidden nodes.

Parameter identifiability of the model with DAG shown in Figure  \ref{fig:kruskal}, is an instance of a more general theorem of
\citet{Kruskal77}.
See also \citep{StegSid,JRkruskal}.
However, known proofs of the full Kruskal theorem do not yield an explicit procedure for recovering
parameters.
Nonetheless, a proof of a restricted theorem (the essential idea of which is not
original to this work, and has been rediscovered several times)
does.
We include this argument for Theorem \ref{thm:Kruskal} below,
since it is still not widely known and provides motivation for
the approach to the proof of Theorem \ref{thm:caseb} for models associated to a second DAG, shown in Figure \ref{fig:second}.
Our analysis of the second model appears to be entirely novel.
For both models, we
characterize the exceptional parameters for which these procedures
fail, giving a precise characterization of a set containing all
non-identifiable parameters.

\subsection{Explicit cases of Kruskal's Theorem}

The model we consider has the DAG of model 3-0 in Table
\ref{tb:results}, also shown in Figure \ref{fig:kruskal} for convenience.
\begin{figure}[h]
\begin{center}
\includegraphics[height=2cm]{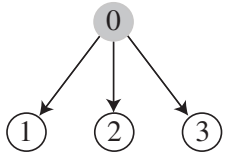}
\end{center}
\caption{The DAG of model 3-0, the Kruskal model} \label{fig:kruskal}
\end{figure}

Parameters for the model are:
\begin{enumerate}
\item $\mathbf p_0=P(X_0)\in  \Delta^{n_0-1}$,  a stochastic vector giving the
distribution for the $n_0$-state hidden variable $X_0$.
\item For each of $i=1,2,3$, a $n_0\times n_i$ stochastic matrix $M_i=P(X_i|X_0)$.
\end{enumerate}

We use the following terminology.
\begin{defn}
The \emph{Kruskal row rank} of a matrix $M$ is the maximal number
$r$ such that every set of $r$ rows of $M$ is linearly independent.
\end{defn}
Note that the Kruskal row rank of a matrix may be less than its rank,
which is the maximal $r$ such that \emph{some} set of $r$ rows is independent.

Our special case of Kruskal's Theorem is the following:

\begin{thm} \label{thm:Kruskal} Consider the model represented by the DAG
of model 3-0, where variables $X_i$ have $n_i \ge 2$ states, with $n_1,n_2\ge n_0$. 
Then generic parameters of the model are identifiable up to
label swapping, and an algebraic procedure for determination of the parameters
from the joint probability distribution $P(X_1,X_2,X_3)$ can be given.

More specifically, if $\mathbf p_0$ has no zero entries, $M_1,M_2$ have
 rank $n_0$, and $M_3$ has Kruskal row rank at least 2, then the parameters can
 be found through determination of the roots of certain $n_0$-th degree univariate
 polynomials and solving linear equations. The coefficients of these polynomials
 and linear systems are rational expressions in the joint distribution.
\end{thm}

\begin{proof} For simplicity, consider first the case $n_0=n_1=n_2=n$.
Let $P=P(X_1,X_2,X_3)$ be a probability distribution of observable variables arising from the model, viewed as a $n\times n\times n_3$ array.

Marginalizing $P$ over $X_3$ (\emph{i.e.}, summing over the 3rd index), we obtain a matrix which, in terms of the unknown parameters, is the matrix product
$$P_{\cdot\cdot+}=P(X_1,X_2)= M_1^T\diag(\mathbf p_0)M_2.$$
Similarly, if $M_3=(m_{ij})$, then the slices of $P$ with  third index fixed at $i$ (\emph{i.e.,} the conditional distributions given $X_i=i$, up to normalization) are
\begin{equation*}P_{\cdot\cdot i}=P(X_1,X_2,X_3=i)=M_1^T\diag(\mathbf p_0)\diag (M_3(\cdot,i))M_2,\end{equation*}
where $M_3(\cdot,i)$ is the $i$th column of $M_3$.

Assuming $M_1, M_2$ are non-singular, and $\mathbf p_ 0$ has no zero entries, $P_{\cdot\cdot+}$ is invertible and we see
\begin{equation}P_{\cdot\cdot+}^{-1}P_{\cdot\cdot i}=M_2^{-1}\diag (M_3(\cdot,i))M_2.\label{eq:eigprod}\end{equation}
Thus the entries of the columns of $M_3$ can be determined (without
order) by finding the eigenvalues of the
$P_{\cdot\cdot+}^{-1}P_{\cdot\cdot i}$, and the rows of $M_2$ can be
found by computing the corresponding left eigenvectors,
normalizing so the entries add to 1. (If $M_3$ has repeated entries
in the $i$th column, the eigenvectors   may not be uniquely
determined. However, since the matrices
$P_{\cdot\cdot+}^{-1}P_{\cdot\cdot i}$ for various $i$ commute, and
$M_3$ has Kruskal row rank 2 or more, the set of these matrices do
uniquely determine a collection of simultaneous 1-dimensional
eigenspaces. We leave the details to the reader.) This determines
$M_2$ and $M_3$, up to the simultaneous ordering of their rows.

A similar calculation with $P_{\cdot\cdot i}P_{\cdot\cdot+}^{-1}$
determines $M_1$, and $M_3$, up to the row order. Since the rows of
$M_3$ are distinct (because it has Kruskal rank 2), fixing some
ordering of them fixes a consistent order of the rows of all of the
$M_i$.

Finally, one determines $\mathbf p_0$ from $M_1^{-T}P_{\cdot\cdot+}M_2^{-1}=\diag(\mathbf p_0)$.

\smallskip
The hypotheses on the rank and Kruskal rank of the parameter matrices can be expressed
through the non-vanishing of minors, so
all assumption on parameters used in this
procedure can be phrased as the non-vanishing of certain
polynomials. As a result, the exceptional set where it cannot be performed
is contained in a proper algebraic subset of the parameter set.

Since the computations to perform the procedure involve computing
eigenvalues and eigenvectors of matrices whose entries are rational
in the joint distribution, the second paragraph of the theorem is
justified.

\smallskip

In the more general case of $n_1,n_2\ge n_0$, one can apply the argument above to $n_0\times n_0\times n_3$ subarrys of $P$ corresponding to submatrices  of $M_1$ and $M_2$ that are invertible. 
All such subarrays  will lead to the same eigenvalues of the matrices analogous to those of equation \eqref{eq:eigprod}, so eigenvectors can be matched up  to reconstruct entire rows of $M_1$ and $M_2$.
The vector $\mathbf p_0$ is determined by a formula similar to that above, using a subarray of the marginalization $P_{\cdot\cdot +}$.
 \end{proof}

\subsection{Another special model}
The model we consider next has the DAG of model 4-3b in Table 
\ref{tb:results}, reproduced in Figure \ref{fig:second} for convenience.
\begin{figure}[h]
\begin{center}
\includegraphics[height=2cm]{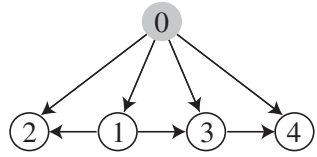}
\end{center}
\caption{The DAG of model 4-3b.} \label{fig:second}
\end{figure}

Parameters for the model are:
\begin{enumerate}
\item $\mathbf p_0=P(X_0)\in  \Delta^{n_0-1}$, a stochastic vector
giving the distribution for the $n_0$-state hidden variable $X_0$.
\item Stochastic matrices $M_1=P(X_1|X_0)$ of size $n_0\times n_1$;
$M_i=P(X_i|X_0,X_1)$ of size $n_0n_1\times n_i$ for $i=2,3$; and $M_4=P(X_4|X_0,X_3)$ of size $n_0n_3\times n_4 $.
\end{enumerate}

\begin{thm} \label{thm:caseb} Consider the model
represented by the DAG of model 4-3b, where variables
$X_i$ have $n_i\ge 2$ states, with $n_2,n_4\ge n_0$. Then generic
parameters of the model are identifiable up to label swapping, and
an algebraic procedure for determination of the parameters from the
joint probability distribution $P(X_1,X_2,X_3,X_4)$ can be given.

More specifically, suppose $\mathbf p_0,M_1,M_3$ have no zero entries, 
the $n_0\times n_2$ and  $n_0\times n_4$ matrices
\begin{align*}
M_2^i&=P(X_2|X_0,X_1=i),\ 1\le i\le n_1,\text{ and}\\
M_4^j&=P(X_4|X_0,X_3=j),\ 1\le j\le n_3
\end{align*}
 have rank $n_0$,  and
 there exists some $i, i'$ with $1\le i<i'\le n_1$ such that for all $1\le j<j'<n_3$, $1\le k<k'\le n_4$
 the entries of $M_3$ satisfy inequality \eqref{eq:badM3} below. Then from the resulting
 joint distribution the parameters can be found through determination of the
 roots of certain $n$-th degree univariate polynomials and solving linear equations.
 The coefficients of these polynomials and linear systems are rational expressions
 in the entries of the joint distribution.
\end{thm}

\begin{proof} Consider first the case $n_0=n_2=n_4=n$.
With $P=P(X_1,X_2,X_3,X_4)$ viewed as an $n_1\times n\times n_3\times n$ array,
we work with $n\times n$ `slices' of $P$, $$P_{i,j}= P(X_1=i,X_2,X_3=j,X_4),$$
\emph{i.e.}, we essentially condition on $X_1,X_3$,  though omit the normalization.

Note that these slices can be expressed as
\begin{equation} P_{i,j}= (M_2^i)^T
D_{i,j}M_4^j,\label{eq:fact} \end{equation}
where
$D_{i,j}=\diag(P(X_0,X_1=i,X_3=j))$ is the diagonal matrix given in terms of parameters by
$$D_{i,j}(k,k)=\mathbf p_0(k) M_1(k,i)M_3((k,i),j),$$
and
$M_2^i$ and  $M_4^j$ are as in the statement of the Theorem.

Equation \eqref{eq:fact} implies for $1\le i,i'\le n_1$ and  $1\le j,j'\le n_3$ that
\begin{equation}P_{i,j}^{-1} P_{i,j'}P_{i',j'}^{-1} P_{i',j }=\\
(M_4^j)^{-1}   D_{i,j}^{-1}D_{i,j'}D_{i',j'} ^{-1} D_{i',j} M_4^j,
\label{eq:prod}\end{equation}
and
the hypotheses on the parameters imply the needed invertibility.
But this shows the rows of $M_4^j$ are left eigenvectors of this product.

In fact, if $i\ne i'$, $j\ne j'$, then the eigenvalues of this product are distinct,
for generic parameters. To see this, note the eigenvalues are
\begin{equation}M_3((k,i),j')M_3((k,i'),j)/( M_3((k,i),j)M_3((k,i'),j')),\label{eq:eigs}
\end{equation}
for $1\le k\le n$, so distinctness of eigenvalues is equivalent to 
\begin{multline}
M_3((k,i),j')M_3((k,i'),j) M_3((k',i),j)M_3((k',i'),j')\\ \ne  M_3((k,i),j)M_3((k,i'),j')M_3((k',i),j')M_3((k',i'),j),\label{eq:badM3}
\end{multline}
for all $1\le k< k'\le n$. 
Thus a generic choice of $M_3$ leads to distinct eigenvalues.

With distinct eigenvalues, the eigenvectors
are determined up to scaling. But since each row of $M_4^j$ must sum to 1,
the rows of $M_4^j$ are therefore determined by $P$.

The ordering of the rows of the $M_4^j$ has not yet been determined.
To do this, first fix an arbitrary ordering of the rows of $M_4^1$, say, which
imposes an arbitrary labeling of the states for $X_0$. Then using equation \eqref{eq:fact},
from $P_{i,1}(M_4^1)^{-1}$ we can determine $D_{i,1}$ and $M_2^i$ with their
rows ordered consistently with
 $M_4^1$.
For $j\ge 1$, using equation
\eqref{eq:fact} again, from $(M_2^i)^{-T}P_{i,j}$ we can determine
$D_{i,j}$ and $M_4^j$ with a consistent row order.
Thus $M_2$ and $M_4$ are determined.

\medskip

To determine the remaining parameters, again appealing to equation \eqref{eq:fact},
we can recover the distribution $P(X_0,X_1,X_2)$ using
 $$(M_2^i)^{-T} P_{i,j}(M_4^j)^{-1}=\diag( P(X_0, X_1=i,X_3=j)).$$
With $X_0$ no longer hidden, it is straightforward to determine the remaining parameters.

The general case of $n_0\le n_2,n_4$, is handled by considering subarrays, just as in the proof of the preceding theorem.
\end{proof}

\begin{rem}
In the case of all binary variables, the expression in \eqref{eq:eigs} is
just the conditional odds ratio
for the observed variables $X_1,X_3$, conditioned on $X_0$. The
inequality \eqref{eq:badM3} can thus be interpreted as saying there is a
non-zero 3-way interaction between the variables $X_0,X_1,X_2$, which
is the generic situation.
\end{rem}

\section{Small binary DAG models}\label{sec:binary}

All variables are assumed binary throughout this section.
In Table \ref{tb:results} of the Appendix,
we list each of the binary DAG models with one latent node which is parental to up to 4 observable nodes.
We number the graphs as $A$-$Bx$ where $A=|O|=|V|-1$ is the number
of observed variables,  $B=|E|-|O|$ is the number of directed edges
between the observed variables, and $x$ is a letter appended to distinguish
between several graphs with these same features. As the table presents
only the case that all variables are binary, the observable distribution lies in a space 
of dimension $2^A-1$.

The primary information in this table is in the column for $k$, indicating the
parameterization map is generically $k$-to-one.
As discussed in the introduction, the existence of such a $k$ is not obvious, and does not follow
from the behavior of general polynomial maps in real variables.

The models 4-3e and 4-3f, for which the parameterization maps are
generically 4-to-one, are particularly interesting cases, as for these models
there are non-identifiability issues that arise neither from overparameterization
(in the sense of a parameter space of larger dimension than the distribution space)
nor from label swapping.
While these models are ones that can plausibly be imagined as being used for
data analysis, they have a rather surprising failure of identifiability, which
is explored more precisely in Section \ref{sec:4to1}.

\medskip

We now turn to establishing the  results in Table
\ref{tb:results}. 

For many of the models $A$-$Bx$ the dimension of the parameter space
computed by equation \eqref{eq:dimtheta} exceeds the dimension
$2^A-1$ of the probability simplex in which the joint distribution
of observed variables lies. In these cases, the following
Proposition applies to show the parameterization is generically
infinite-to-one. We omit its proof for brevity.

\begin{prop}\label{prop:infinite}
Let $f:S \to \RR^m$ be any map defined by real polynomials, where
$S$ is an open subset of $\RR^n$ and $n>m$. Then $f$ is generically
infinite-to-one.
\end{prop}

This proposition applies to all models in Table \ref{tb:results}
with an infinite-to-one parameterization, with the single exception
of 4-2a. For that model, amalgamating $X_1$ and $X_2$ together, and
likewise  $X_3$ and $X_4$, we obtain a model with two 4-state
observed variables that are conditionally independent given a binary
hidden variable $X_0$. One can show that the probability distributions for this model
forms  an 11-dimensional object,
and then a variant of Proposition \ref{prop:infinite} applies.

For models  3-0 and 4-3b (and the Markov equivalent 4-3a), specializing Theorems \ref{thm:Kruskal} and \ref{thm:caseb}
of the previous section to binary variables
yields the claims in the table.

For the remaining models, the strategy is to first marginalize or condition on an
observable variable
to reduce the model to one already understood. One then attempts to
`lift' results on the reduced model back to the original one.

We consider in detail only some of the models, indicating how the
arguments we give can be adapted to others with minor modifications.

\subsection{Model 4-1}

\begin{figure}[h]
\begin{center}
\includegraphics[height=2cm]{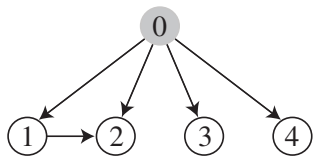} 
\end{center}
\caption{The DAG of model 4-1.}\label{fg:4.1}
\end{figure}

Referring to Figure \ref{fg:4.1}, since node 2 is a sink, marginalizing over $X_2$ gives an instance
of model 3-0 with the same parameters,
after discarding $P(X_2|X_0,X_1)$.
Thus generically all parameters except
$P(X_2|X_0,X_1)$ are determined, up to label swapping.

But note that if the (unknown) joint distribution of
$X_0,X_1,X_2,X_3$ is written as an $8\times 2$ matrix $U$, with $$U
((i,j,k),\ell)=P(X_0=\ell, X_1=i,X_2=j,X_3=k ),$$ and
$M_4=P(X_4|X_0)$, then the matrix product  $UM_4$ has entries
$$(UM_4)((i,j,k),\ell)=P(X_1=i,X_2=j,X_3=k, X_4=\ell),$$
which form the observable joint distribution. Since generically
$M_4$ is invertible, from the observable distribution and each of
the already identified label swapping variants of $M_4$ we can find $U$. From $U$ we
marginalize to obtain $P(X_0,X_1,X_2)$ and $P(X_0,X_1)$. Under the
generic condition that $P(X_0), P(X_1|X_0)$ are strictly positive,
$P(X_0,X_1)$ is as well, and so we can compute
$P(X_2|X_0,X_1)=P(X_0,X_1,X_2)/P(X_0,X_1)$.

Models 4-0 and 4-2d are handled similarly, by marginalizing over the sink nodes 4 and 3, respectively.

An alternative argument for model 4-1  and 4-0 proceeds by
amalgamating the observed variables, $X_1,X_2$, into a single
4-state variable, and applying Theorem \ref{thm:Kruskal} directly to that
model. We leave the details to the reader.

\subsection{Models 4-2b,c} Up to renaming of nodes, the DAGs for models 4-2b and 4-2c are Markov equivalent.
Thus by Theorem \ref{thm:equivDAGs}, it is enough to consider model
4-2c, as shown in Figure \ref{fg:4.2c}.

\begin{figure}[h]
\begin{center}
\includegraphics[height=2cm]{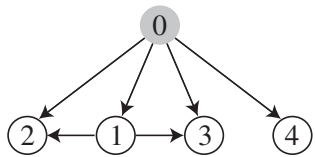} 
\end{center}
\caption{The DAG of model 4-2c.}\label{fg:4.2c}
\end{figure}

We condition on $X_1=j$, $j=1,2$ to
obtain two related models. Letting $X_i^{(j)}$ denote the
conditioned variable at node $i$, the resulting observable
distributions  are
\begin{multline*}
P(X_2^{(j)},X_3^{(j)},X_4^{(j)})=P(X_2,
X_3,X_4~|~ X_1=j)\\=P(X_1=j)^{-1} P(X_1=j,X_2, X_3,X_4).
\end{multline*}
With a hidden variable
$ X_0^{(j)}$ and observed variables $ X_2^{(j)},  X_3^{(j)},
X_4^{(j)}$, these distributions arise from a DAG like that of
model 3-0.  With parameters for the original model $\mathbf
p_0=P(X_0)$, $2\times 2$
 matrices $M_i=P(X_i|X_0)$ for $i=1,4$, and $2\times 4$ matrices
 $M_i=P(X_i~|~X_0,X_1)$, $i=2,3$
and ${\mathbf e_j}$ the standard basis vector,
parameters for the conditioned models are:
\begin{enumerate}
\item the vector
\begin{align*}\mathbf p_0^{(j)}&=P(X_0^{(j)})= P(X_0|X_1=j)\\
 &=P(X_1=j)^{-1} P(X_0,X_1=j)\\
 &= \frac 1{\mathbf p_0^TM_1\mathbf e_j}(\diag(\mathbf p_0)M_1\mathbf e_j),\end{align*}
\item the $2\times 2$ stochastic matrix $M_4^{(i)}=P(X_4^{(i)}|X_0^{(i)})=M_4$,
and
\item for $i=2,3$, the $2\times 2$ stochastic matrix
$M_i^{(j)}=P( X_i^{(j)}| X_0^{(j)})$, whose rows are the $(0,j)$ and $(1,j)$ rows  of $M_i$.
\end{enumerate}

Now if $\mathbf p_0$ and column $j$ of $M_1$ have non-zero entries,
it follows that $\mathbf p_0^{(j)}$ has no zero entries. If
additionally $ M_2^{(j)},  M_3^{(j)}, M_4$ all have rank 2,  by
Theorem \ref{thm:Kruskal} the parameters of these conditioned models
are identifiable, up to the labeling of the states of the hidden
variable. As these assumptions are generic conditions on the
parameters of the original model, we can generically identify the
parameters of the conditioned models.

In particular, $M_4$ can be identified up to reordering its rows, and is invertible. But let
$U$ denote the (unknown) $8\times2$ matrix with
$U((i,j,k),\ell)=P(X_0=\ell,X_1=i, X_2=j,X_3=k)$. Then $P=UM_4$, has
as its entries the observable distribution $P(X_1,X_2,X_3,X_4)$.
Thus $U=PM_4^{-1}$ can be determined from $P$. Since $U$ is the
distribution of the induced model on $X_0,X_1,X_2,X_3$ with no
hidden variables, it is then straightforward to identify all
remaining parameters of the original model.

Thus all parameters are identifiable generically, up to label swapping. More specifically,
they are identifiable provided that for either $j=0$ or $1$ the
three matrices $M_4,M_2^{(j)}, M_3^{(j)}$ have rank 2, and $\mathbf
p_0$ and the $j$th column of $M_1$ have non-zero entries.

\subsection{Models 4-3e,f}\label{subsec:43ef} Due to Markov
equivalence, we need
consider only 4-3e, as shown in Figure \ref{fg:4.3e}.

\begin{figure}[h]
\begin{center}
\includegraphics[height=2cm]{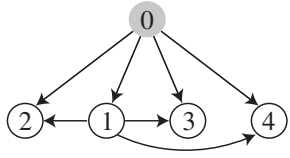} 
\end{center}
\caption{The DAG of model 4-3e.}\label{fg:4.3e}
\end{figure}

By conditioning on  $X_1=j$, $j=1,2$, we obtain two models of the
form of {3-0}. One checks that the induced parameters for these
conditioned models are generic. Indeed, in terms of the original
parameters they are $P(X_i~|~X_0,X_1=j)$, $i=2,3,4$, which are
generically non-singular since they are simply submatrices of the
$P(X_i~|~X_0,X_1)$, and at the hidden node $$P(X_0~|~X_1=j)=\frac{
P(X_1=j~|~X_0)P(X_0)}{
\sum_\ell P(X_1=j ~|~X_0=\ell)P(X_0=\ell)}$$ which generically has
non-zero entries.

Thus for generic parameters on the original model, up to label swapping 
we can determine $P(X_0~|~X_1=j)$ and $P(X_i~|~X_0,X_1=j)$, 
$i=2,3,4$.
However, we do not have an ordering of the states of $X_0$ that is
consistent for the recovered parameters for the two models. 
Generically we have 4 choices of parameters for the 2 models taken
together. Each of these 4 choices leads to a possible joint
distribution $P(X_0,X_1)$; viewing this joint distribution as a
matrix, the 4 versions differ only by independently interchanging
the two entries in each column, thus keeping the same
marginalization $P(X_1)$. Generically, each of the 4 distributions for $P(X_0,X_1)$ yields different parameters $P(X_0)$
and $P(X_1|X_0)$. The matrices $P(X_i|X_0,X_1)$, $i=2,3,4$, are then
obtained using the same rows as in $P(X_i~|~X_0,X_1=j)$, though the
ordering of the rows is dependent on the choice made previously.

Having obtained 4 possible parameter choices, it is straightforward
to confirm that they all lead to the same joint distribution. Thus
the parameterization map is generically 4-to-one.

\section{Identification of causal effects}\label{sec:4to1}

Here we examine the impact of $k$-to-one model parameterizations on the causal effect of one observable variable on another, when a latent variable acts as a confounder.
For simplicity we assume that the latent variable is binary, though our discussion can be extended to a more general setting.

According to Theorem 3.2.2 (Adjusting for Direct Causes), p.~73 of \cite{Pearl2nd}, 
the causal effect of $X_i$ on $X_j$ can be obtained from model parameters by an appropriate sum over the states of the other direct causes of $X_j$.  
This sum is invariant under a relabeling of the states of those direct causes, and therefore the causal effect 
is not affected by label swapping if one of these is latent. As an instance, the causal effect of $X_1$ on $X_2$ in model 
4-2b is:
\begin{multline}
P(X_2  \mid  do(X_1=x_1))=P(X_2  \mid  X_1=x_1, X_0=1)P(X_0=1)\\
 +P(X_2  \mid  X_1=x_1, X_0=2) P(X_0=2).
\label{do}
\end{multline}
Thus when label swapping is the only source of parameter non-identifiability, causal effects are uniquely determined by the observable distribution.
 
Things are more complex when parameter non-identifiability arises in other ways. For example, model 4-3e has
one binary latent variable but a $4$-to-$1$ parameterization.
In Table \ref{tb:rational} two choices of parameters, (1), (2), are given for this model, as well as the common observable distribution they produce. 
These parameters and their two variants from label swapping at node 0 give the four elements of the fiber of the observable distribution.

\begin{table}
\caption{A rational example for model 4-3e. The parameter choices (1) and (2) lead to the same observable distribution, shown at the bottom.
For the $4\times 2$ matrix parameters, row indices refer to states of a pair of parents $i<j$ ordered 
lexicographically as  (1,1), (1,2), (2,1), (2,2), with the first entry referring to parent $i$, and the second to parent $j$.}\label{tb:rational}
\begin{tabular}{lc}
\hline
(1)& \\
&$\mathbf{p}_0 = \begin{pmatrix} 2/5 & 3/5 \end{pmatrix}
\hskip .5cm
M_1 = \begin{pmatrix} 2/5 & 3/5\\ 14/15 &1/15 \end{pmatrix}
$\\
\ \\
&$M_2 =  \begin{pmatrix} 
2/5 & 3/5 \\ 3/5 & 2/5 \\ 4/5 & 1/5 \\ 9/10 & 1/10
\end{pmatrix}
\hskip .5cm 
M_3 =  \begin{pmatrix} 
1/5& 4/5\\  9/20 & 11/20\\ 1/2 & 1/2\\  2/5 &  3/5
\end{pmatrix}
\hskip .5cm 
M_4 =  \begin{pmatrix} 
1/2 & 1/2\\  7/10 &  3/10\\ 4/5 & 1/5\\  3/5 &  2/5
\end{pmatrix}$\\
\ \\
\hline
(2)& \\
&$\mathbf{p}'_0 = \begin{pmatrix} 1/5 & 4/5 \end{pmatrix}
\hskip .5cm
M'_1 = \begin{pmatrix} 4/5 & 1/5\\ 7/10 & 3/10 \end{pmatrix}
$\\
\ \\
& $M'_2 =  \begin{pmatrix} 
2/5 & 3/5\\ 9/10 & 1/10\\ 4/5 & 1/5\\  3/5 &  2/5
\end{pmatrix}
\hskip .5cm 
M'_3 =  \begin{pmatrix} 
1/5 & 4/5\\  2/5 &  3/5\\ 1/2 &  1/2\\ 9/20 & 11/20
\end{pmatrix}
\hskip .5cm 
M'_4 =  \begin{pmatrix} 
1/2 &  1/2\\  3/5 &   2/5\\ 4/5 &  1/5\\ 7/10 &   3/10
\end{pmatrix}
$\\
\ \\
\hline
\ \\
\multicolumn{2}{l}
{$P(X_1,X_2,X_3=1,X_4=1) = \hskip .5in  P(X_1,X_2,X_3=1,X_4=2) =$}\\
\multicolumn{2}{c}{$\left[ \begin{matrix} 116/625 & 34/625\\  27/500 & 39/1250 \end{matrix} \right] \hskip .8in
\left[ \begin{matrix} 32/625 & 13/625\\ 63/2500 & 17/1250  \end{matrix} \right]$}\\
\ \\
\multicolumn{2}{l}{$P(X_1,X_2,X_3=2,X_4=1) = \hskip .5in P(X_1,X_2,X_3=2,X_4=2)  = $} \\
\multicolumn{2}{c}{$\left[ \begin{matrix} 128/625 & 52/625\\171/2500 & 24/625 \end{matrix} \right] \hskip .8in
\left[ \begin{matrix} 44/625 & 31/625\\ 81/2500 & 21/1250 \end{matrix} \right]$}\\
\ \\
\hline
\end{tabular}
\end{table}
%

For any parameters of model 4-3e, the causal effect of $X_1$ on $X_2$ is again as given in equation (\ref{do}). However, due to the 
$4$-to-$1$ parameterization, there can be two different causal effects that are consistent with an observable distribution. As such, 
there may be distributions such that one causal effect leads to the conclusion that 
there is a positive effect of $X_1$ on $X_2$ 
(i.e., setting $X_1=2$ gives a higher probability of $X_2=2$ than setting $X_1=1$), while the 
other causal effect leads to the conclusion that there is a negative effect of $X_1$ on $X_2$. 
Indeed, the observable distribution in Table \ref{tb:rational} is such an instance.
In Table \ref{tb:doeffect} the two causal effects corresponding to that given distribution
are shown. Here parameters (2) lead to a positive effect of $X_1$ on $X_2$, while parameters (1) lead to a negative effect.

\begin{table}
 \caption{The causal effects of $X_1$ on $X_2$ for the example in 
 Table \ref{tb:rational}.}\label{tb:doeffect}
\begin{tabular}{l|c|c|c}
\hline
Parameters & (a) & (b) & (a)$-$(b)\\ 
 &$\dsp P(X_2 =2 \mid do(X_1=2) )$ & $\dsp P(X_2 =2 \mid do(X_1=1) )$ \\
 \hline
 & & &\\
\quad (1) &  ${11}/{50}$& $ {9}/{25}$ & ${-7}/{50}$\\
& & & \\
\quad (2) &  ${17}/{50}$& $ {7}/{25}$ & ${\phantom{-}3}/{50}$\\
& & & \\
\hline
\end{tabular}
\end{table}

More generally, for generic observable distributions of this model, choices of parameters that differ other than by label swapping will give different causal effects.  However, it varies whether the effects have the same or different signs.

\section{Conclusion}

Paraphrasing \citet{Pearl2012}, the problem of identifying causal
effects in non-parametric models has been ``placed to rest"  by the
proof of completeness of the $do$-calculus and related graphical criteria.
In this paper we show that the introduction of modest (parametric) assumptions 
on the size of the 
state spaces of variables allows for identifiability of parameters that otherwise would be
non-identifiable.
Causal effects can be computed from identified parameters, if desired, 
but our techniques allow for the recovery of all parameters.  
In the process of proving parameter identifiability for several small networks, we use techniques 
inspired by a theorem of Kruskal, and other novel approaches. This framework can be applied 
to other models as well.

We have at least three reasons to extend the work described in this paper. The first 
is to develop new techniques and to prove new theoretical results for parameter
identifiability; this provides the foundation of our work.  
A second is to reach the stage at which one can easily determine parameter identifiability for DAG models with hidden variables that are 
used in statistical modeling; this motivates our work.
A third and  related focus of future work is to address the scalability of our 
approach and to automate it.  As noted above many of our arguments do not depend on variables being binary. Also,
a strategy that we used successfully to handle larger models is to 
first marginalize or condition on an observable variable to reduce the model 
to one already understood, and then to `lift' results on the reduced model back 
to the original one.  We are working towards turning this strategy into an algorithm.

\subsubsection*{Acknowledgments}

The authors thank the American Institute of Mathematics, where this work was begun during a workshop on
Parameter Identification in Graphical Models, and continued through AIM's SQuaRE program.

\appendix
\section*{Appendix}
Table \ref{tb:results} shows all DAGs with 4 or fewer observable nodes
and one hidden node that is a parent of all observable ones. See Section
\ref{sec:binary} for model naming convention. Markov equivalent
graphs appear on the same line.
The dimension of the parameter space is $\dim(\Theta)$,  and $2^A-1$
is the dimension of the probability simplex in which the joint distribution lies.
The parameterization map is generically $k$-to-one.

\begin{center}

\begin{longtable}{|c|c|c|c|c|}
\caption{Small binary DAG models.}\label{tb:results} \\
\hline \multicolumn{1}{|c|}{Model} & \multicolumn{1}{c|}{Graph} & \multicolumn{1}{c|}{$\dim(\Theta)$}& \multicolumn{1}{c|}{$2^A-1$}&\multicolumn{1}{c|}{$k$} \\ \hline
\endfirsthead
\multicolumn{5}{c}
{{\tablename\ \thetable{} -- continued from previous page}} \\
\hline \multicolumn{1}{|c|}{Model} & \multicolumn{1}{c|}{Graph} & \multicolumn{1}{c|}{$\dim(\Theta)$}& \multicolumn{1}{c|}{$2^A-1$}&\multicolumn{1}{c|}{$k$} \\ \hline
\endhead
\hline \multicolumn{5}{r}{Continued on next page} \\
\endfoot
\hline \hline
\endlastfoot
2-$B$, $B\ge 0$&&$\ge5$&3&$\infty$\\
\hline\hline
3-0& \includegraphics[height=.07\textheight]{fg3-0.eps}&7&7&$2$\\
\hline
3-$Bx$, $B\ge 1$&&$\ge9$&7&$\infty$\\
\hline \hline
4-0&\includegraphics[height=.07\textheight]{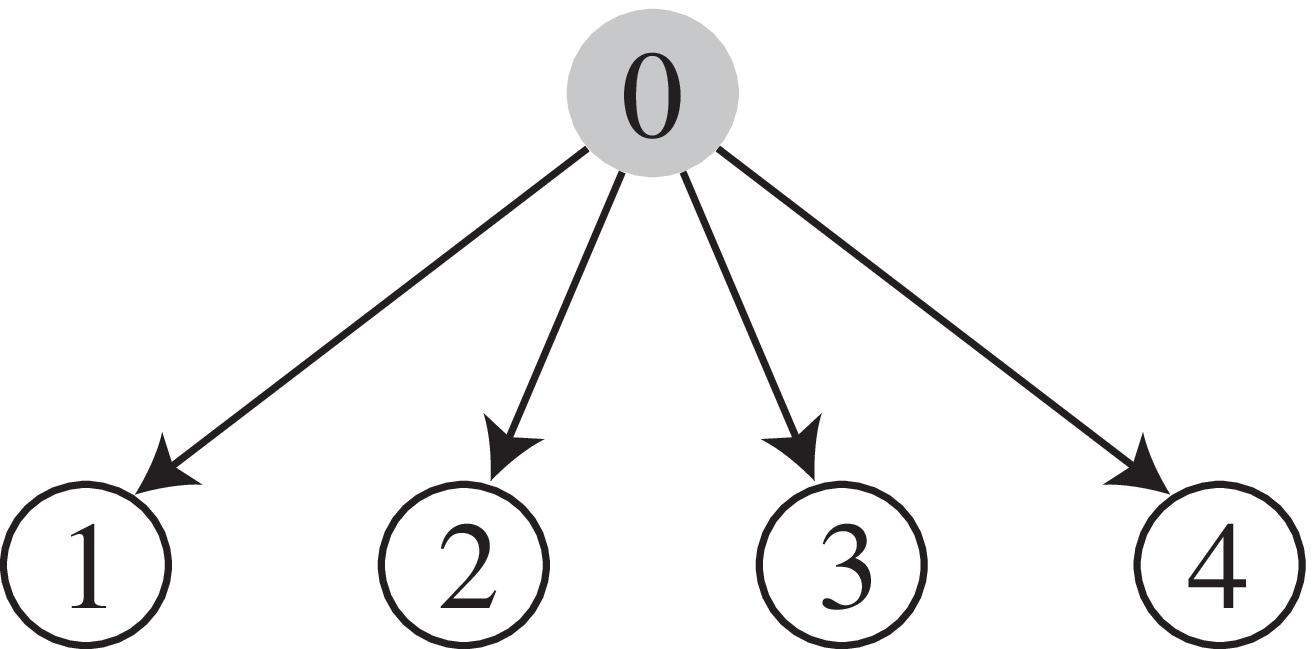} &9&15&$2$\\
\hline
4-1&\includegraphics[height=.07\textheight]{fg4-1.eps} &11&15&$2$\\
\hline

4-2a& \includegraphics[height=.07\textheight]{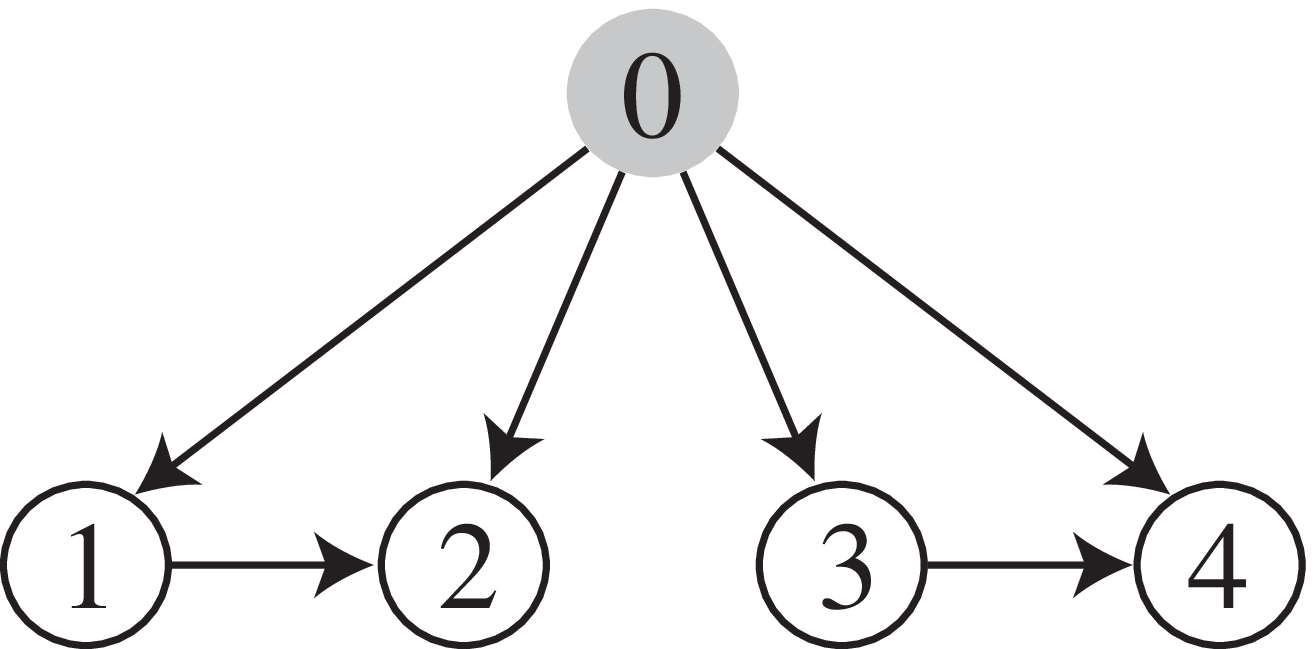}&13&15&$\infty$\\
\hline
4-2b,c&\includegraphics[height=.07\textheight]{fg4-2b.eps}, \includegraphics[height=.07\textheight]{fg4-2c.eps} &13&15&$2$\\
\hline
4-2d&  \includegraphics[height=.07\textheight]{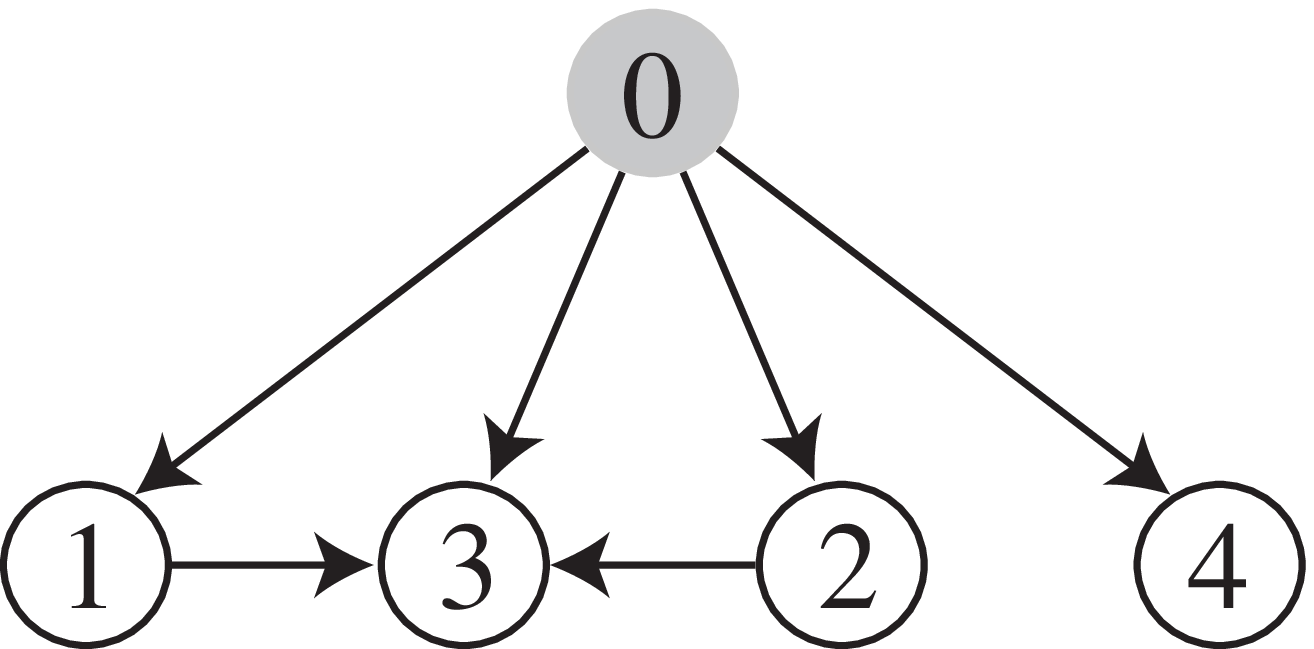} &15&15&$2$\\
\hline
4-3a,b& \includegraphics[height=.07\textheight]{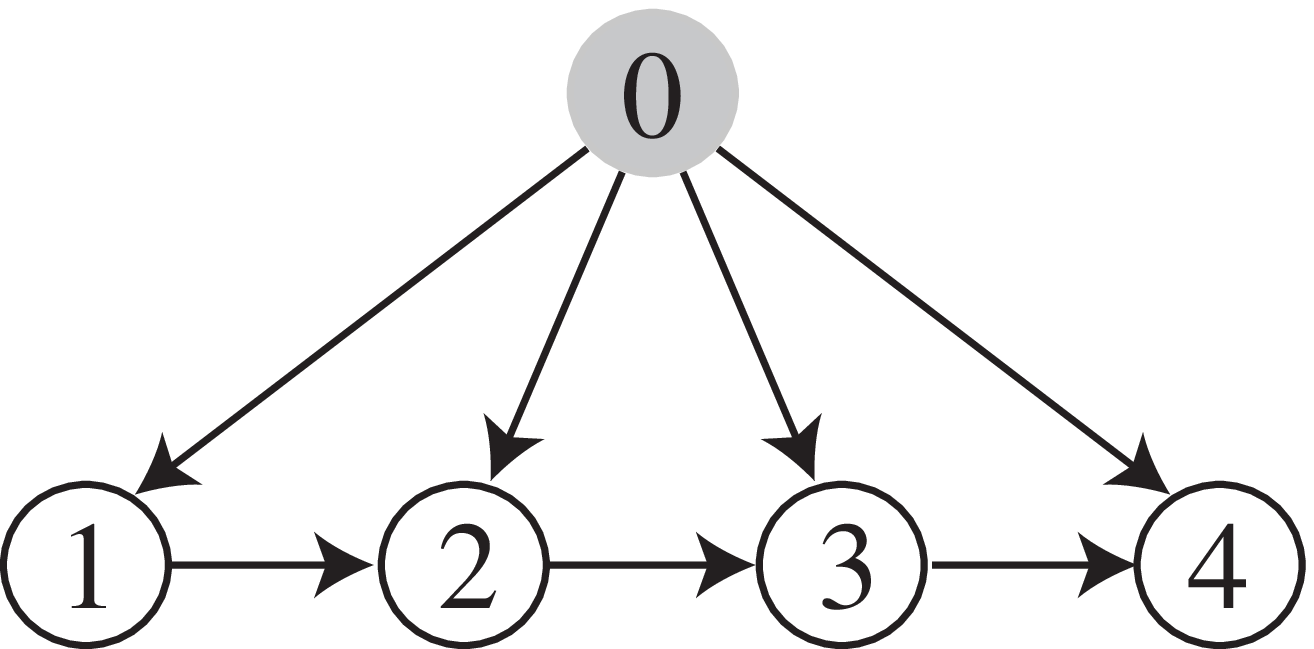}, \includegraphics[height=.07\textheight]{fg4-3b.eps} &15&15&$2$\\ 
\hline
4-3c,d& \includegraphics[height=.07\textheight]{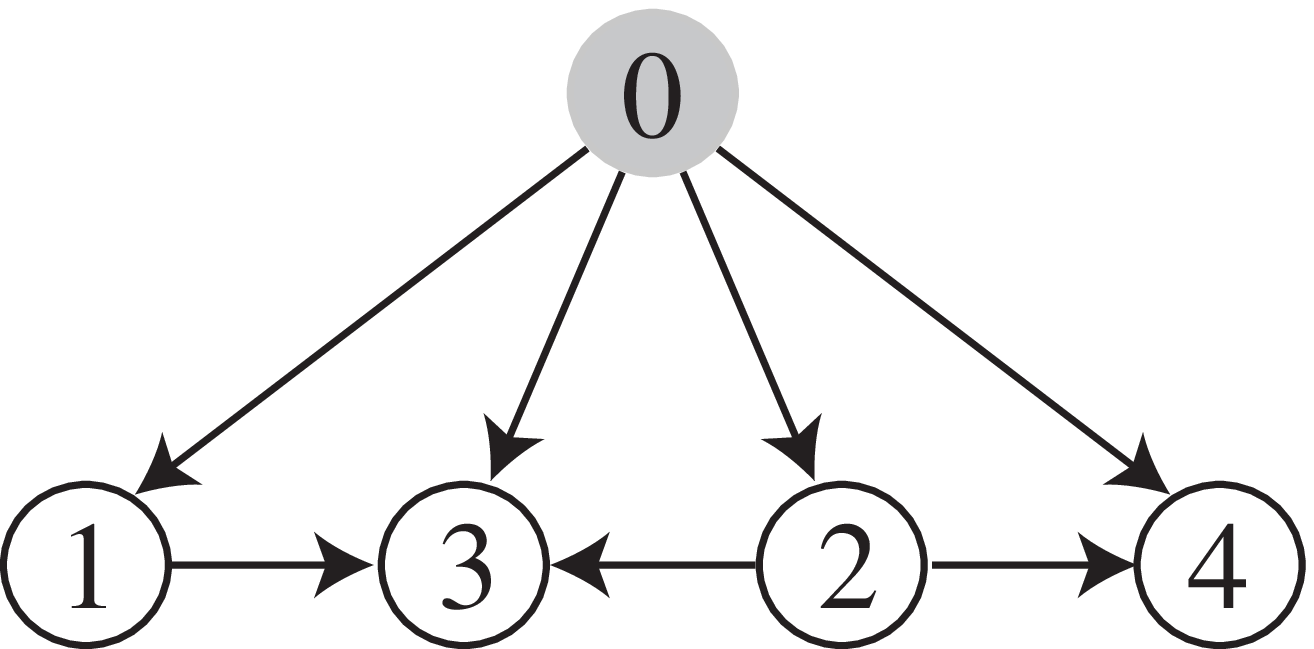}, \includegraphics[height=.07\textheight]{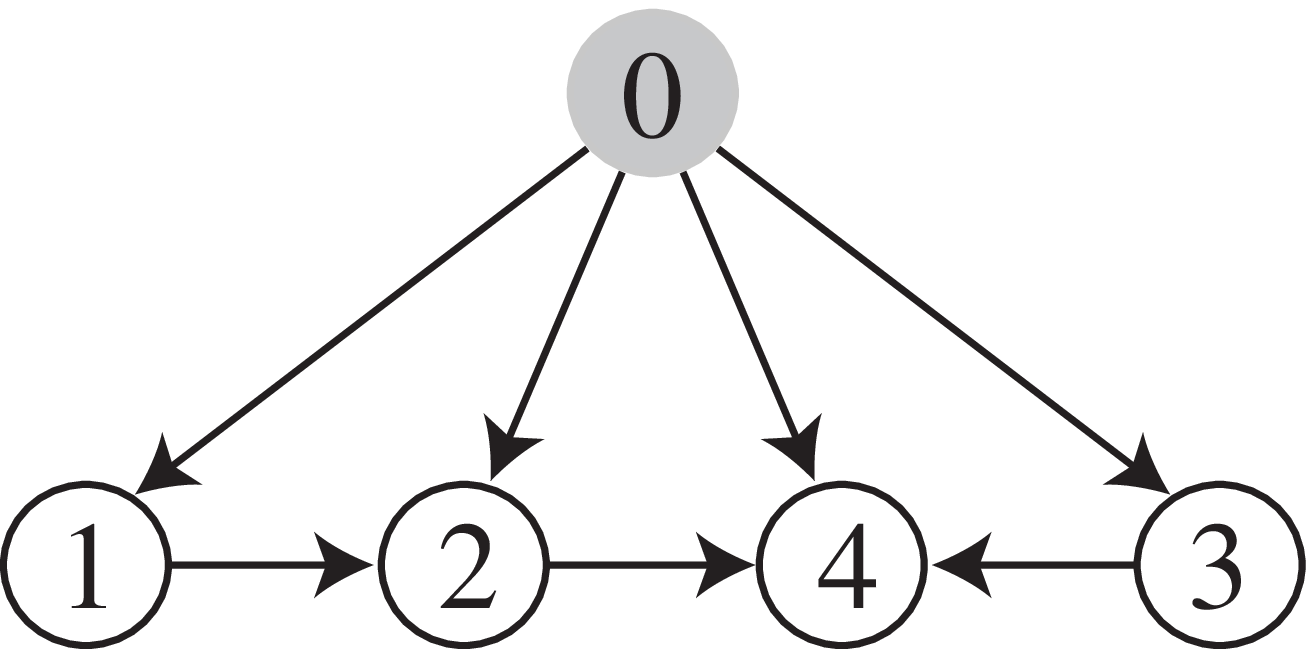}&17&15&$\infty$\\
\hline
4-3e,f &\includegraphics[height=.07\textheight]{fg4-3e.eps}, \includegraphics[height=.07\textheight]{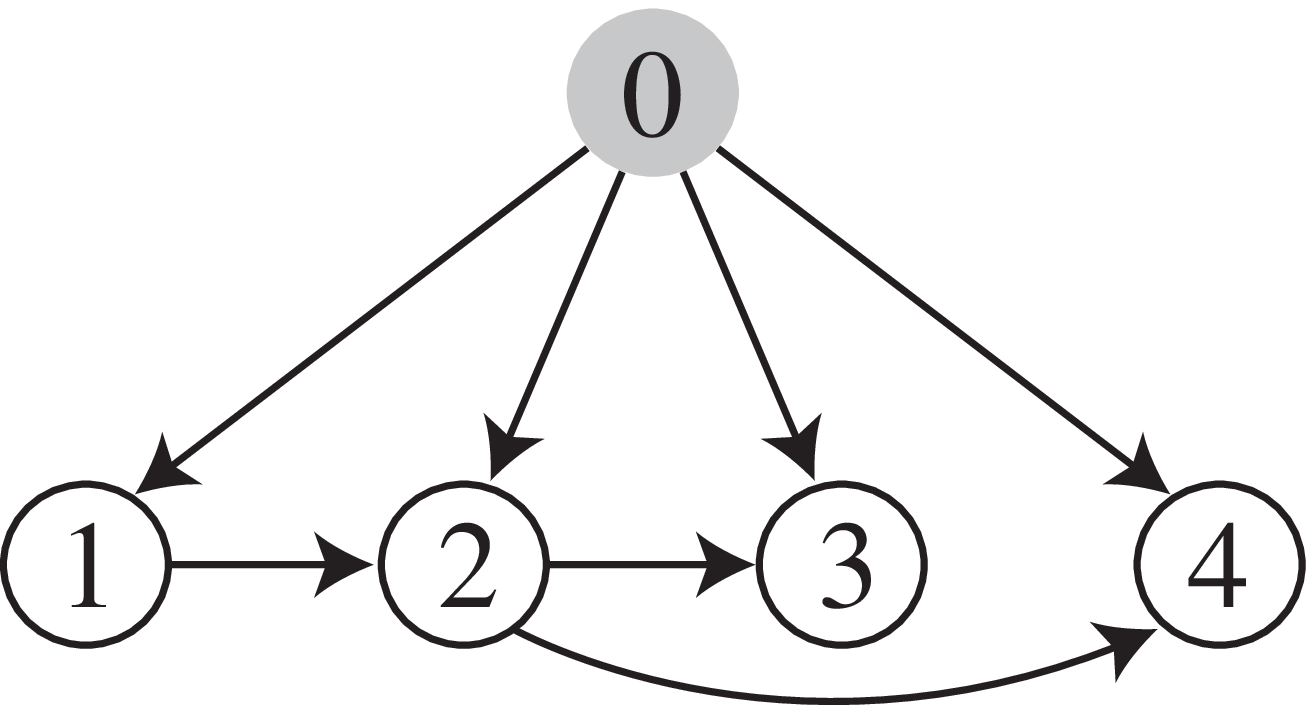} &15&15&$4$\\
\hline
4-3g & \includegraphics[height=.07\textheight]{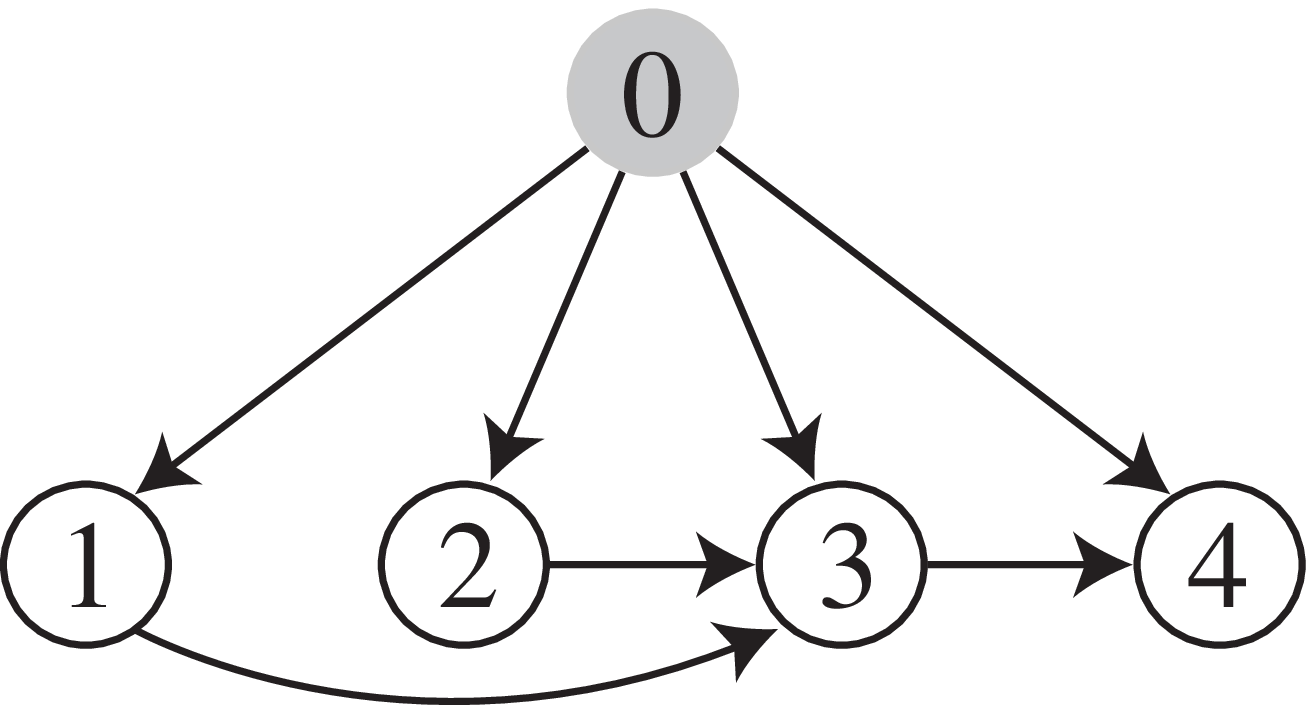}&17&15&$\infty$\\
\hline
4-3h & \includegraphics[height=.07\textheight]{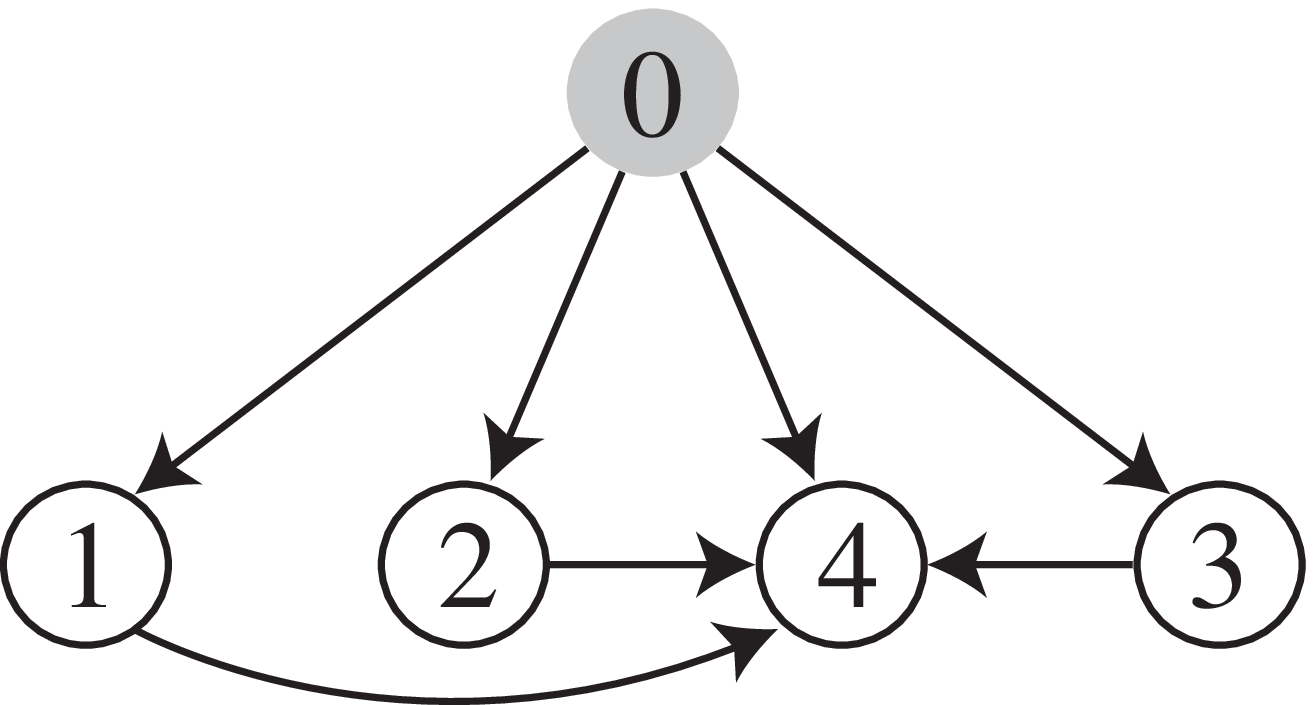} &25&15&$\infty$\\
\hline
4-3i &  \includegraphics[height=.07\textheight]{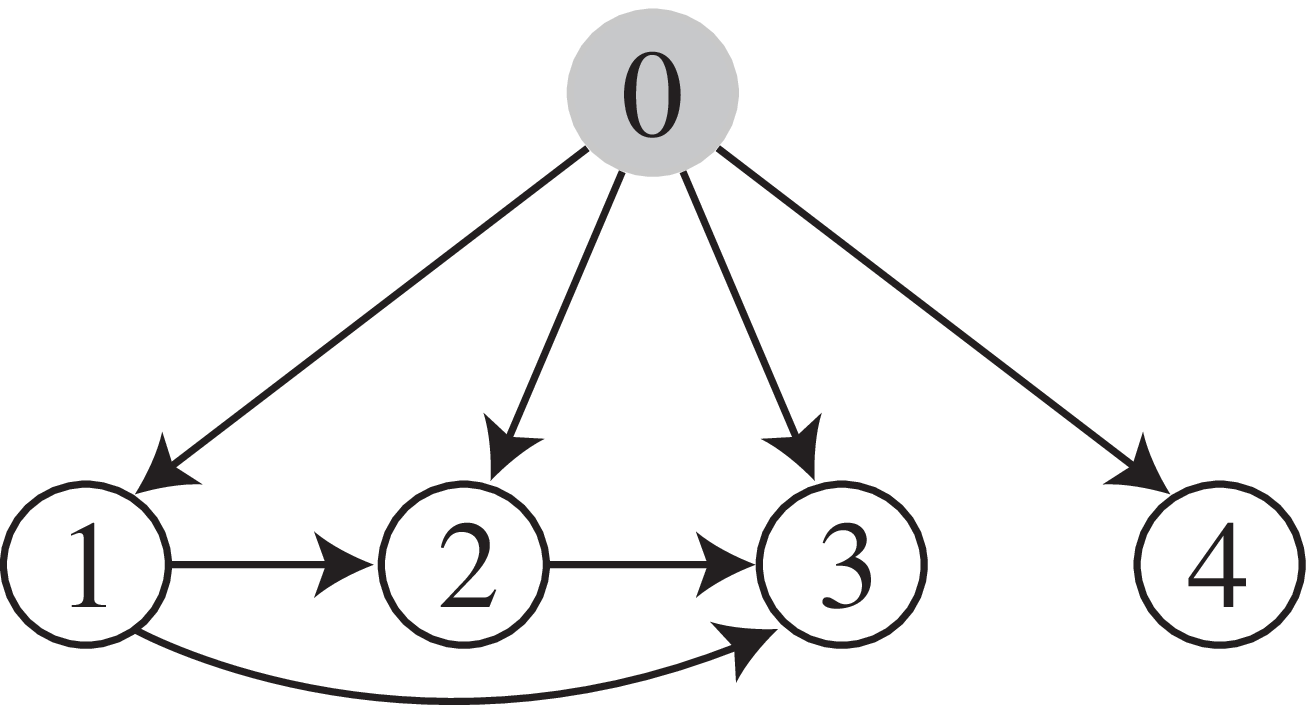}&25&15&$\infty$\\
\hline
4-$Bx$, $B\ge 4$ & &$\ge19$&15&$\infty$
\end{longtable}

\end{center}

\bibliographystyle{DeGruyter}
\bibliography{Causal}

\Addresses

\end{document}